\documentclass[12pt]{amsart}

\usepackage{amsfonts,amssymb,stmaryrd,amscd,amsmath,latexsym,amsbsy,enumerate,bbold}

\usepackage[margin=1in]{geometry}

\usepackage[bookmarks, bookmarksdepth=2, colorlinks=true, linkcolor=blue, citecolor=blue, urlcolor=blue]{hyperref}

\newtheorem{theorem}{Theorem}[section]
\newtheorem{lemma}[theorem]{Lemma}
\newtheorem{proposition}[theorem]{Proposition}
\newtheorem{corollary}[theorem]{Corollary}
\theoremstyle{definition}
\newtheorem{definition}[theorem]{Definition}

\newtheorem{question}[theorem]{Question}
\newtheorem{conjecture}[theorem]{Conjecture}
\newtheorem{remark}[theorem]{Remark}

\newcommand{\Hom}{\text{Hom}}

\newcommand{\Rep}{\text{Rep}}

\newcommand{\h}{\mathfrak{h}}

\newcommand{\ben}{\begin{enumerate}}
\newcommand{\een}{\end{enumerate}}

\newcommand{\CC}{{\mathbb{C}}}
\newcommand{\ZZ}{{\mathbb{Z}}}

\newcommand{\arxiv}[1]{\href{http://arxiv.org/abs/#1}{{\tt arXiv:#1}}}

\theoremstyle{plain}

\newtheorem*{sol}{Solution}

\theoremstyle{definition}

\theoremstyle{remark}

\newcommand{\solu}[1]{\begin{sol}{\bf (\ref{#1})}}

\begin{document}

\title[On Cohen--Macaulayness of $S_n$-invariant subspace arrangements]{On Cohen--Macaulayness of $S_n$-invariant\\subspace arrangements}
\date{June 13, 2015}

\author{Aaron Brookner}
\address{Department of Mathematics, Massachusetts Institute of Technology,
Cambridge, MA 02139, USA}
\email{brookner@mit.edu}

\author{David Corwin}
\address{Department of Mathematics, Massachusetts Institute of Technology,
Cambridge, MA 02139, USA}
\email{corwind@mit.edu}

\author{Pavel Etingof}
\address{Department of Mathematics, Massachusetts Institute of Technology,
Cambridge, MA 02139, USA}
\email{etingof@math.mit.edu}

\author{Steven V Sam}
\address{Department of Mathematics, University of California, Berkeley, 
Berkeley, CA, 94720, USA}
\email{svs@math.berkeley.edu}

\maketitle

\begin{abstract}
Given a partition $\lambda$ of $n$, consider the subspace $E_\lambda$ of $\CC^n$ where the first $\lambda_1$ coordinates are equal, the next $\lambda_2$ coordinates are equal, etc. In this paper, we study subspace arrangements $X_\lambda$ consisting of the union of translates of $E_\lambda$ by the symmetric group. In particular, we focus on determining when $X_\lambda$ is Cohen--Macaulay. This is inspired by previous work of the third author coming from the study of rational Cherednik algebras and which answers the question positively when all parts of $\lambda$ are equal. We show that $X_\lambda$ is not Cohen--Macaulay when $\lambda$ has at least $4$ distinct parts, and handle a large number of cases when $\lambda$ has $2$ or $3$ distinct parts. Along the way, we also settle a conjecture of Sergeev and Veselov about the Cohen--Macaulayness of algebras generated by deformed Newton sums. Our techniques combine classical techniques from commutative algebra and invariant theory; in many cases we can reduce an infinite family to a finite check which can sometimes be handled by computer algebra.
\end{abstract}

\setcounter{tocdepth}{1}
\tableofcontents

\section{Introduction} 

Let $\lambda=(\lambda_1, \dots ,\lambda_r)$ be a partition of $n$. Then we have a subspace $E_\lambda$ in $\Bbb C^n$
defined by the equations 
$$
x_1=\cdots=x_{\lambda_1},\quad x_{\lambda_1+1}=\cdots=x_{\lambda_1+\lambda_2},\quad \dots,\quad  x_{n-\lambda_r+1}=\cdots=x_n.
$$
Let $S_n$ be the group of permutations of $n$ elements, and 
\[
X_\lambda=S_n\cdot E_\lambda
\]
be the union of the $S_n$-translates of $E_\lambda$. 
Then $X_\lambda$ is an algebraic variety with an action of $S_n$. Note that $\dim X_\lambda=r$, and $X_\lambda=X_\lambda^0\times \Bbb C$, where 
\[
X_\lambda^0=\lbrace{x\in X_\lambda \mid \sum x_i=0\rbrace}
\]
and $\dim X_\lambda^0=r-1$. 

The goal of this paper is to address the following two questions.   

\begin{question}\label{whencm} (i) For which $\lambda$ is the variety $X_\lambda$ Cohen--Macaulay (CM)? 
 
(ii) For which $\lambda$ is the variety $X_\lambda/S_n$ CM? 
\end{question} 

Clearly, in these questions we can replace $X_\lambda$ with $X_\lambda^0$ (this does not change the questions). 
Also, it is clear that if the answer to Question \ref{whencm}(i) is ``yes'' for some partition $\lambda$, then 
the answer to Question \ref{whencm}(ii) for this partition is also ``yes''.  

Questions \ref{whencm}(i),(ii) seem to be difficult to treat by standard methods of commutative algebra (see \cite{schenck-sidman} for a survey of the commutative algebra of subspace arrangements; we also point out \cite{geramita-weibel, Re, yuzvinsky} for treatments of the Cohen--Macaulayness property of subspace arrangements).
They are motivated by the paper \cite{EGL}, which gives a full answer to Question \ref{whencm}(i) in the special case 
$\lambda=(m^r,1^s)$. The argument of \cite{EGL} uses representation theory of Cherednik algebras, 
while an argument based on classical methods is unknown. This is so even in the case $r=1$, addressed in \cite{BGS}. 
Also, Question \ref{whencm}(ii) for partitions of the form $(m^r,p^s)$ is asked in \cite{SV}. 

We also point out the paper \cite{devadas-sam}, in which the Cohen--Macaulay property of $X_\lambda$ for the special case $\lambda = (m,1^s)$ is used to calculate character formulas for modular irreducible representations of Cherednik algebras. A curious property of these subspace arrangements is that they are defined over $\ZZ$ (and flat over $\ZZ$) but the Cohen--Macaulay property can depend on characteristic (see \cite[Example 5.2]{BGS}). It would be of interest to understand what is failing in small characteristics, but in this paper we only work over characteristic $0$.

\bigskip

While the full answer to either of the Questions \ref{whencm}(i),(ii) is unknown to us, we answer them in many cases, combining 
the methods of representation theory of Cherednik algebras, standard methods of commutative algebra (such as Reisner's theorem, \cite{Re}), and computer calculations in Macaulay2 \cite{macaulay2}.  
In particular, we show that the answer to Questions \ref{whencm}(i),(ii) 
is ``no'' if $\lambda$ has at least four distinct parts, or three distinct parts such that the largest one is not equal to the sum of two smaller ones. 
We also formulate some conjectures for the remaining cases based on computational evidence. Finally, we
settle a conjecture from \cite{SV} (end of Section 4) stating that the algebra generated by deformed Newton sums 
\[
a(y_1^i+\cdots+y_r^i)+(z_1^i+\cdots+z_s^i), \qquad i=1,2,\dots
\]
is CM for generic complex $a$. Finally, we discuss the set of values of $a$ for which this property fails. 

Our main conclusion is that the CM property is rather rarely satisfied for $X_\lambda$ and $X_\lambda/S_n$, and whenever it is, 
there is often some structure behind it, coming from representation theory and integrable systems, which is instrumental in the proof of 
the CM property. Namely, our main results are summarized by the following theorem. 

\begin{theorem}\label{maintheo} 
\begin{enumerate}[\rm 1.]
\item If all parts of $\lambda$ are equal, i.e., $\lambda=(m^r)$, then $X_\lambda$ and hence $X_\lambda/S_n$ is CM (Theorem \ref{cher}). 

\item
\begin{enumerate}[\rm (a)]
\item  If $\lambda=(m^r,1^s)$, where $m>1,r>0,s>0$, then $X_\lambda$ and $X_\lambda/S_n$ are CM if and only if either
$m=2$ or $s<m$ (Theorem \ref{cher}, Proposition \ref{adapt3.11}). 

\item Let $\lambda=(m^r,p^s)$, $m>p>1,r>0,s>0$, and let $m=kb$, $p=kc$, with $k\in \Bbb Z_{>0}$ and $\gcd(b,c)=1$. If $c\le r$, $3\le b\le s$, 
then $X_\lambda/S_n$ and hence $X_\lambda$ is not CM (Proposition \ref{adapt3.11}). 

\item If $c=1$ and $b=2$ or $b>s$ then $X_\lambda/S_n$ 
is CM (Theorem \ref{cher}). 

\item For fixed $r>0,s>0$, $X_\lambda/S_n$ is CM for all but finitely many pairs $(b,c)$ (Theorem \ref{SVconj}). 

\item $X_{(4,2,2)}$ is CM, while $X_{(3^r,2^s)}$ is not CM 
for $r\ge 1,s\ge 1,r+s\ge 3$ and $X_{(5^r,2^s)}$ is not CM 
for $r\ge 1$, $s\ge 2$ (Proposition \ref{422}).  
\end{enumerate}
 
\item 
\begin{enumerate}[\rm (a)]
\item If $\lambda$ has three distinct parts $a>b>c$ and $a\ne b+c$ then $X_\lambda/S_n$ and hence $X_\lambda$ is not CM
(Theorem \ref{noCMinv}). 

\item If $\lambda=(b+c,b^r,c^s)$ then $X_\lambda$ is not CM (Theorem \ref{noCM}).  
\end{enumerate}
 
\item If $\lambda$ has at least four distinct parts, then $X_\lambda/S_n$ and hence $X_\lambda$ is not CM (Theorem \ref{noCMinv}). 
\end{enumerate}
\end{theorem}

Unfortunately, the theorem still leaves open infinite families of cases when $\lambda$ has $2$ or $3$ distinct parts. We give some conjectures for next steps in \S\ref{sec:nonCM}.

\subsection*{Supporting files}

A few of the proofs in this paper reduce to a finite computation, which we perform with the computer algebra system Macaulay2 \cite{macaulay2}. We have included the scripts to perform these calculations as ancillary files in the {\tt arXiv} submission of this paper. In the proofs where they are used, we have included a description of the calculations being performed.

\subsection*{Acknowledgments} 
P. E. is grateful to A. Polishchuk for discussions that led to this research, and to 
A. Veselov and M. Feigin for explanations concerning the paper \cite{SV}. We also thank Eric Rains who verified Propositions~\ref{bplus1b} and~\ref{aplus1s2} in this paper using MAGMA. 
The work of P. E. was  partially supported by the NSF grant DMS-1000113. The work of A. B. and D. C. 
was done in the Summer Program of Undergraduate Research at the Mathematics department of MIT. 
The work of S. S. is supported by a Miller research fellowship.

\section{Preliminaries and known results} 

Recall that a finitely generated commutative $\Bbb C$-algebra $R$ (or, equivalently, the scheme $X={\rm Spec}(R)$) 
is Cohen--Macaulay (CM) if $R$ is a finitely generated (locally) free module over some polynomial subalgebra $\Bbb C[u_1,\dots,u_r]\subset R$. 
In this case, $R$ is a (locally) free module over any such subalgebra over which it is finitely generated (as a module). 
For more details about the CM property see \cite[Chapter 18]{Eis}.  

The following theorem was proved in \cite[Proposition 3.11]{EGL}:

\begin{theorem}\label{cher} Suppose that $\lambda=(m^r,1^s)$
(i.e., $r$ copies of $m$ and $s$ copies of $1$). Then $X_\lambda$ is CM if and only if either 
$m>s$, or $m\le 2$. 
\end{theorem} 

The proof is based on the representation theory of rational Cherednik algebras and Reisner's theorem, \cite{Re}. 

Also, note that if $\lambda=(p,q)$ then $X_\lambda$ is automatically CM, since $X_\lambda^0$ is $1$-dimensional. 

Let us now consider the quotient $X_\lambda/S_n$. The algebra of functions on this variety 
is described by the following proposition. 

\begin{proposition}\label{funconinv}
The algebra ${\mathcal O}(X_\lambda/S_n)$ is the subalgebra 
of $\Bbb C[y_1,\dots,y_r]$ generated by the Newton $\lambda$-sums
$$
P_{i,\lambda}(y_1,\dots,y_r):=\lambda_1 y_1^i+\cdots+\lambda_r y_r^i.
$$
\end{proposition} 

\begin{proof} It is easy to see that the map $E_\lambda\to X_\lambda/S_n$ is surjective, so 
the corresponding pullback defines an inclusion of ${\mathcal O}(X_\lambda/S_n)$ into $\Bbb C[E_\lambda]$. 
The ring $\Bbb C[E_\lambda]$ can be realized as $\Bbb C[y_1,\dots,y_r]$, where $y_1$ is the value of $x_i$ for $1\le i\le \lambda_1$,
$y_2$ is the value of $x_i$ for $\lambda_1<i\le \lambda_1+\lambda_2$, etc.  
Now, ${\mathcal O}(X_\lambda/S_n)$ is generated by ordinary power (=Newton) sums 
$p_i=x_1^i+\cdots+x_n^i$. Upon the embedding into $\Bbb C[y_1,\dots,y_r]$, we get $p_i(x)=P_{i,\lambda}(y)$. 
This implies the statement. 
\end{proof}  

\begin{corollary}\label{multiple} $X_\lambda/S_n\cong X_{k\lambda}/S_{kn}$ for any $k\ge 1$, where $k\lambda:=(k\lambda_1,\dots,k\lambda_r)$.
Thus, $X_\lambda/S_n$ is CM if $\lambda=((2k)^r,k^s)$ or $\lambda=((mk)^r,k^s)$ 
if $s<m$. 
\end{corollary}

\begin{proof} 
The first statement follows from Proposition \ref{funconinv}, as 
$P_{i,k\lambda}=kP_{i,\lambda}$. The second statement follows from the first one and 
Theorem \ref{cher}.  
\end{proof} 

\begin{remark} 
The variety $Y_\lambda$ in the case $\lambda=(2^2,1^{n-4})$ is considered in \cite{KMSV}. Also, in this paper the authors propose the problem to study $Y_\lambda$ 
for general $\lambda$. 
\end{remark}

Proposition~\ref{funconinv} motivates the following definition. 

\begin{definition} Let $\lambda=(\lambda_1,\dots,\lambda_r)$ be a collection of variables. 
The algebra $A_r$ of Newton $\lambda$-sums is the subalgebra of 
$\Bbb C[\lambda_1^{\pm 1},\dots,\lambda_r^{\pm 1},y_1,\dots,y_r]$ generated over $\Bbb C[\lambda_1^{\pm 1},\dots,\lambda_r^{\pm 1}]$ by 
$P_{i,\lambda}$ for $i=1,2,\dots$ 
\end{definition} 

For any subset $S\subset \lbrace{1,\dots,r\rbrace}$, let $\lambda_S=\sum_{i\in S}\lambda_i$. 
Let $A_{r,\rm loc}=A_r[\lambda_S^{-1},S\ne \emptyset]$ be the localization of $A_r$ 
obtained by inverting $\lambda_S$ for all $S\ne \emptyset$. 

The following proposition is a straightforward generalization of \cite[Theorem 5]{SV}. 

\begin{proposition}\label{fingen}
The algebra $\Bbb C[\lambda_i,\lambda_S^{-1},y_1,\dots,y_r]$ and hence its subalgebra $A_{r,\rm loc}$ are finitely generated modules 
over the subalgebra $\Bbb C[\lambda_i,\lambda_S^{-1},P_{i,\lambda}, i=1,\dots,r]$. In particular, 
$A_{r,\rm loc}$ is a finitely generated algebra. The same statements hold for $\Bbb C[y_1,\dots,y_r]$ and the specialization 
$A_\lambda$ of $A_r$ at any $\lambda\in \Bbb C^r$ such that $\lambda_S\ne 0$ for any nonempty $S$.  
\end{proposition}  

\begin{proof} It suffices to show that if $\lambda_S\ne 0$ for nonempty $S$ then the system of equations 
\[
\lambda_1y_1^i+\cdots+\lambda_ry_r^i=0, \qquad i=1,\dots,r
\]
has only the zero solution. 
Suppose $(y_1,\dots,y_r)$ is a solution. Let $S_j$ be the level sets of the function $y_k$ of $k$, 
and $y_k=z_j$ if $k\in S_j$ (so that $z_j$ are distinct). Then our equations become
\[
\lambda_{S_1}z_1^i+\cdots+\lambda_{S_m}z_m^i=0, \qquad i=1,\dots,r.
\]
Since $m\le r$, by the Vandermonde determinant formula 
this implies that $\lambda_{S_j}z_j=0$ for all $j$. Since $\lambda_{S_j}\ne 0$ for any $j$, we get that $z_j=0$ for all $j$, as desired.  
\end{proof} 

\begin{remark}
\begin{enumerate}
\item If $\lambda_i$ are positive integers (enumerated in decreasing order), then the fact that $A_\lambda$ is a finitely generated algebra 
also follows from the fact that $A_\lambda={\mathcal O}(X_\lambda/S_n)$ and Hilbert's theorem on invariants. 

\item We do not know the smallest $N$ for which the polynomials $P_{i,\lambda}$ for $i\le N$ generate $A_\lambda$ for generic $\lambda$ (as a module over $\Bbb C[P_{1,\lambda},\dots, P_{r,\lambda}]$ or as a ring). 

\item  If $\lambda_S=0$ for some $S\ne \emptyset$ (but $\lambda_i\ne 0$ for all $i$) then $A_\lambda$ is not finitely generated. 
To see this, let us first see that $A_{1,-1}$ is not finitely generated. 
This algebra is generated by the polynomials $F_i:=y^i-z^i$, $i=1,2,\dots$. 
Setting $y=z+w$, we find that $F_i=z^{i-1}w+ \cdots$. So in each monomial in $F_i$, 
the $z$-degree is at most $i-1$ times the $w$-degree, which implies that $F_{i+1}$ 
cannot be expressed via $F_1,\dots,F_i$, i.e., $A_{1,-1}$ is infinitely generated. 
Now, if $\lambda_{i_1}+\cdots+\lambda_{i_k}=0$, we can set 
$x_{i_1}=x_{i_2}=\cdots=x_{i_{k-1}}=y$, $x_{i_k}=z$, and all other $x_i=0$, 
and obtain $A_{1,-1}$ as a quotient of $A_\lambda$, 
which implies the statement. 

For a similar reason, in this case $A_\lambda$ is not Noetherian (the infinite chain of ideals $I_m:=(F_1,F_2,\dots,F_m)$ in $A_{1,-1}$ is 
strictly ascending) and thus not CM.
\end{enumerate}
\end{remark} 

Assume that $\lambda_S\ne 0$ for $S\ne \emptyset$, let $Y_\lambda={\rm Spec}(A_\lambda)$ be the variety corresponding to $A_\lambda$, 
and let $S_\lambda$ be the subgroup of $S_r$ 
preserving $\lambda$. Then by Proposition \ref{fingen}, we have a natural finite morphism $\pi \colon \Bbb C^r/S_\lambda\to Y_\lambda$, 
induced by the embedding $A_\lambda\subset \Bbb C[y_1,\dots,y_r]$. It is easy to see that $\pi$ is surjective (as it is finite and dominant), and birational 
(i.e., is a normalization morphism).\footnote{To see that $\pi$ is birational, note that when $y_i$ are distinct, the function $\sum_{i=0}^\infty P_{i,\lambda}t^i=\sum_j\frac{\lambda_j}{1-y_jt}$ determines $y_j$ uniquely up to the action of $S_\lambda$.} On the other hand, $\pi$ is not always injective, e.g., for $r=3$ and $\lambda=(3,2,1)$ (so that $S_\lambda=1$), the points $(t,0,0)$ and $(0,t,t)$ in $\Bbb C^3$ have the same image under $\pi$ (as the polynomials $P_{i,\lambda}$ have the same values at these points).   

The following question generalizes Question \ref{whencm}(ii):

\begin{question}\label{whencm1} For which $\lambda\in \Bbb C^n$ with $\lambda_S\ne 0$, $S\ne \emptyset$ is the variety $Y_\lambda$ (or, equivalently, the algebra $A_\lambda$) CM? 
\end{question} 

Let $A_\lambda^0$ be the quotient of $A_\lambda$ by the ideal generated by $P_{1,\lambda}$, and $Y_\lambda^0={\rm Spec}(A_\lambda^0)$. By using the variables 
$$
\bar y_i=y_i-\frac{\sum \lambda_iy_i}{\sum \lambda_i}
$$
we see that $Y_\lambda=Y_\lambda^0\times\Bbb C$, so $Y_\lambda$ is CM if and only if so is $Y_\lambda^0$. 
Thus, for $n=2$, $Y_\lambda$ is always CM, since $Y_\lambda^0$ is a curve. 

\section{The formal neighborhoods method}

In spite of the results of the previous section, it turns out that $X_\lambda$ and 
even $X_\lambda/S_n$ is rather rarely CM. Let us discuss a method of proving that 
these varieties are not CM for particular $\lambda$ -- the method of formal neighborhoods. 

This method is adopted from the proof of \cite[Proposition 3.11]{EGL}. Namely, arguing as in this proof, one can show: 

\begin{proposition}\label{formneigh} Suppose that $\ell$ cannot be nontrivially written as a sum of some parts of $\lambda$ 
(for example, $\ell\le \lambda_r$). Then if $X_\lambda\subset \Bbb C^n$ is not CM, then $X_{\lambda\cup \ell}\subset \Bbb C^{n+\ell}$ is not CM either.   
The same applies to $X_\lambda/S_n$ and $X_{\lambda\cup \ell}/S_{n+\ell}$. 
\end{proposition}  

\begin{proof} Consider the point $x$ in $X_{\lambda\cup \ell}$ given by $x=(1,\dots,1,0,\dots,0)$, where the number of ones is $\ell$. It is easy to see that the formal neighborhood of $x$ in $X_{\lambda\cup\ell}$ looks like the product of the formal neighborhood of zero in $X_\lambda$ with a formal disk (since, by the assumption on $\lambda$ and $\ell$,  if the coordinates of $x$ vary so that the point remains in $X_{\lambda\cup \ell}$ then the first $\ell$ coordinates must remain equal). The same applies to $X_\lambda/S_n$ and $X_{\lambda\cup \ell}/S_{n+\ell}$. This implies the statement since a local ring is CM if and only if its completion is CM \cite[Proposition 18.8]{Eis}.
\end{proof} 

Let us now adapt the formal neighborhoods method to the variety $Y_\lambda$ for complex $\lambda$ (with $\lambda_S\ne 0$ for $S\ne \emptyset$). 
Let $\lambda'=(\lambda_1,\dots,\lambda_{r-1})$. 

\begin{proposition}\label{formneighy} Assume that $\lambda_r$ is not a sum of a subset of $\lambda_1,\dots\lambda_{r-1}$. Then the formal neighborhood of 
the point $\pi(0,\dots,0,1)$ in $Y_\lambda$ is isomorphic to the product of the formal neighborhood of $\pi(\bold 0)$ in 
$Y_{\lambda'}$ with a $1$-dimensional formal disk. Thus, if $Y_{\lambda'}$ is not CM then 
$Y_\lambda$ is not CM either.   
\end{proposition} 

\begin{proof} Let $y_r=1+u$, and let $y_1,\dots,y_{r-1},u$ be formal variables. 
We have equations 
$$
P_{i,\lambda'}(y_1,\dots,y_{r-1})=Z_i-\lambda_r((1+u)^i-1),
$$
where $Z_i:=P_{i,\lambda}(y_1,\dots,y_r)-\lambda_r$, formal functions on $Y_\lambda$ near $\pi(0,\dots,0,1)$ 
vanishing at that point. We claim that from these equations we can express $u$ as a formal series 
of the $Z_i$. Indeed, by Proposition \ref{fingen}, for some $N$ we have 
$$
P_{N,\lambda'}=F(P_{1,\lambda'},\dots,P_{N-1,\lambda'}),
$$
where $F$ is a quasi-homogeneous polynomial of degree $N$ (with $\deg P_{i,\lambda'}=i$).
Thus, 
$$
Z_N-\lambda_r((1+u)^N-1)=F(Z_1-\lambda_ru,\dots,Z_{N-1}-\lambda_r((1+u)^{N-1}-1)).
$$
This equation has the form $Z_N-\lambda_r Nu=\cdots$, where $\cdots$ stands for quadratic and higher terms in $u$ and $Z_i$. 
Such an equation can clearly be solved for $u$ in the form of a power series in the $Z_i$, as claimed (note that we have not yet used the condition on $\lambda$). 

Now it remains to note that if $\lambda_r$ is not a sum of a subset of $\lambda_1,\dots,\lambda_{r-1}$, 
then the equation $\pi(\bold x)=\pi(0,\dots,0,1)$ has a unique solution $\bold x=(0,\dots,0,1)$. Indeed, near this point the 
value of $y_r$ can be uniquely read off from the function 
$$
\sum_{i=0}^\infty P_{i,\lambda}(y)t^i=\sum_{j=1}^r \frac{\lambda_j}{1-y_jt}.
$$
This together with the above implies the statement.  
\end{proof} 

\begin{remark}
Note that if $\lambda_i$ are positive integers, Proposition \ref{formneighy} reduces to Proposition \ref{formneigh} for $X_\lambda/S_n$. 
\end{remark} 

\section{The algebra of deformed Newton sums}

Let $a$ be a variable. 

\begin{definition} The algebra $\Lambda_{r,s}$ of deformed Newton sums 
is the subalgebra of 
\[
\Bbb C[a^{\pm 1},y_1,\dots,y_r,z_1,\dots,z_s]
\]
generated over $\Bbb C[a^{\pm 1}]$ by the {\it deformed Newton sums} 
$$
Q_{r,s,i}:=a(y_1^i+\cdots+y_r^i)+(z_1^i+\cdots+z_s^i). 
$$
The localized algebra $\Lambda_{r,s,\rm loc}$ is the localization $\Lambda_{r,s,\rm loc}:=\Lambda_{r,s}\otimes_{\Bbb C[a^{\pm 1}]}K$, where 
\[
K:=\Bbb C[a^{\pm 1},(a+p/q)^{-1} \mid 1\le p\le s,\ 1\le q\le r].
\] 
\end{definition}

\begin{proposition}[{\cite[Theorem 5]{SV}}] \label{fingen1} 
The algebra $K[y_1,\dots,y_r,z_1,\dots,z_s]$ and 
hence its subalgebra $\Lambda_{r,s,\rm loc}$ are finitely generated modules 
over $K[Q_{r,s,i}, i=1,\dots,r+s]$. In particular, $\Lambda_{r,s,\rm loc}$
 is a finitely generated algebra. The same statements hold for $\Bbb C[y_1,\dots,y_r,z_1,\dots,z_s]$ and the 
specialization $\Lambda_{r,s,a}$ of $\Lambda_{r,s}$ at any $a\in \Bbb C^*$ such that 
$a\ne -p/q$, $1\le p\le s$, $1\le q\le r$.
\end{proposition} 

\begin{proof} This is a special case of Proposition \ref{fingen}, for $\lambda=(a^r,1^s)$.  
\end{proof} 

\begin{remark}
Note that $\Lambda_{r,s,a}=A_{(a^r,1^s)}$. 
\end{remark}

The following theorem was conjectured by Sergeev and Veselov in \cite{SV} (end of \S 4). 

\begin{theorem}\label{SVconj} The algebra $\Lambda_{r,s,a}$ is Cohen--Macaulay for all but finitely many values of $a\in \Bbb C$. 
\end{theorem}

\begin{proof} By Proposition \ref{fingen1}, $\Lambda_{r,s,a}$ has the same Hilbert series 
for almost all values of $a$. Also, by Theorem \ref{cher}, it is CM for integers $a>s$
(as in this case by Proposition \ref{funconinv}, $\Lambda_{r,s,a}={\mathcal O}(X_\lambda/S_n)$, where 
$\lambda=(a^r,1^s)$). Since in flat families, CM-ness is an open condition, we conclude that 
$\Lambda_{r,s,a}$ is CM for all but finitely many values of $a$, as desired.    
\end{proof} 

\begin{remark}\label{hilser} It is shown in \cite[Theorem 3]{SV} that the Hilbert series of $\Lambda_{r,s,a}$ for generic $a$ is 
$$
h_{r,s}(t)=\frac{1}{(1-t)\cdots(1-t^r)} \sum_{i=0}^s \frac{t^{i(r+1)}}{(1-t)\cdots(1-t^i)}.
$$
Therefore, by Theorem \ref{SVconj}, for generic $a$ the Hilbert series of the generators 
of $\Lambda_{r,s,a}$ as a module over $\Bbb C[Q_{r,s,i},i=1,\dots,r+s]$ 
is 
$$
\widetilde{h}_{r,s}(t)=(1-t^{r+1})\cdots(1-t^{r+s})\sum_{i=0}^s \frac{t^{i(r+1)}}{(1-t)\cdots(1-t^i)}.
$$ 
For instance, if $s=1$, we get 
\[
\widetilde{h}_{r,1}(t)=(1-t^{r+1})(1+\frac{t^{r+1}}{1-t})=1+t^{r+2}+\cdots+t^{2r+1}
\]
(see \cite{SV}, end of \S 4). 
\end{remark} 

Now we would like to study the set 
\[
B(r,s)\subset \Bbb C^*\setminus\lbrace{-p/q \mid 1\le p\le s,\ 1\le q\le r\rbrace}
\]
of ``exceptional'' values of $a$, 
for which $\Lambda_{r,s,a}$ fails to be CM (clearly, they may be nonempty only if $r+s\ge 3$). It is clear that $B(r,s)=B(s,r)^{-1}$, as 
$\Lambda_{r,s,a}=\Lambda_{s,r,a^{-1}}$.

\begin{proposition}\label{incl} If $a\in B(r,s)$ and $a\ne -\frac{p}{r+1}$ with $1\le p\le s$ then $a\in B(r+1,s)$. 
Similarly, if $a\in B(r,s)$ and $a\ne -\frac{s+1}{q}$ with $1\le q\le r$ then $a\in B(r,s+1)$. 
\end{proposition}

\begin{proof} Since $B(r,s)=B(s,r)^{-1}$, it suffices to prove the first statement. 
This follows from Proposition \ref{formneighy} unless $a=k$ for integer $1\le k\le s$. But in this case, we know from Theorem \ref{cher} that 
$a\in B(r+1,s)$ unless $k=1$ or $k=2$, while $a=1,2$ do not belong to $B(r,s)$. This implies the statement.  
\end{proof} 

\begin{proposition}\label{adapt3.11}
Let $a=p/q$, where $p,q$ are coprime positive integers, $p>q$, $3\le p\le s$, $q\le r$, 
Then $a\in B(r,s)$ and $a^{-1}\in B(s,r)$. 
\end{proposition} 

\begin{proof} By Proposition \ref{incl}, it suffices to prove the statement when $r=q$ and $s=p$, so that $a=s/r$, $r<s$, $s\ge 3$. 
In this case, $\Lambda_{r,s,a}=A_\lambda$, where $\lambda=(s^r,r^s)$, a partition of $n=2rs$. 
Consider the formal neighborhood of the point $(0,\dots,0,1,\dots,1)$ of $X_\lambda/S_n$ where the number of zeros and the number of ones are both equal to $rs$. 
Let us generically deform this point inside $X_\lambda/S_n$. This deformation can happen in two ways: 
either the zeros deform into $r$-tuples of $s$ equal coordinates and ones deform into $s$-tuples of $r$ equal coordinates, or the other way. 
This gives two subspaces of dimension $r+s$ in $\Bbb C^{2rs}$ whose intersection is 2-dimensional (the space of vectors whose first $rs$ and last $rs$ coordinates are equal). 
Thus the formal neighborhood in question looks like a 2-dimensional formal disk times 
the formal neighborhood of the intersection point in the transversal intersection 
of two spaces of dimension $r+s-2$. This is not CM by Reisner's theorem, \cite{Re} (or by direct verification) 
if $r+s-2\ge 2$, i.e., for $s\ge 3$. This implies the statement. 
\end{proof} 

This implies that the union of the sets $B(r,s)$ over all $r,s$ contains all positive rational numbers except $1,2,1/2$. 

Let $\overline{B}(r,s)\subset B(r,s)$ be the set of points provided by Proposition \ref{adapt3.11}.
For example:
$$
\overline{B}(2,1)=\overline{B}(1,2)=\emptyset, \quad
\overline{B}(3,1)=\lbrace{1/3\rbrace}, \quad
\overline{B}(2,2)=\emptyset, \quad
\overline{B}(1,3)=\lbrace{3\rbrace}, 
$$
$$
\overline{B}(4,1)=\lbrace{1/3,1/4\rbrace}, \quad
\overline{B}(3,2)=\lbrace{1/3,2/3\rbrace}, \quad
\overline{B}(3,2)=\lbrace{3,3/2\rbrace}, \quad
\overline{B}(1,4)=\lbrace{3,4\rbrace},
$$
etc. 

\begin{conjecture}\label{badset} One has $B(r,s)=\overline{B}(r,s)$. 
\end{conjecture} 

In particular, this would imply that $B(r,s)\subset \Bbb Q$ (from abstract nonsense, it is only clear that $B(r,s)\subset \overline{\Bbb Q}$). 

Computations in Mathematica have confirmed this conjecture for $r+s\le 5$; moreover, for $r+s=3$ we will prove theoretically below that $B(r,s)=\emptyset$.  
 
\begin{remark}\label{exce} 
While $1,2,\frac{1}{2}$ do not belong to the sets $B(r,s)$ (as the corresponding varieties are CM), computations show that at these points 
(for $r\ge 2$ for $1/2$ and $s\ge 2$ for 2) the coefficients of the Hilbert series of $\Lambda_{r,s,a}$ are smaller than those for generic $a$, so they also should be viewed as exceptional. 

On the other hand, for the values $a=-p/q$, $1\le p\le s$, $1\le q\le r$ (which we excluded), \cite[Theorem 2]{SV} implies that $\Lambda_{r,s,a}$ has the same Hilbert series as generically, 
but it is not finitely generated and not Noetherian, so not CM. 

Thus, the full exceptional set is of the form $\widetilde{B}(r,s)=\lbrace{b\in \Bbb C^*|b\ne \pm \frac{p}{q}, 1\le p\le s, 1\le q\le r\rbrace}$. 
\end{remark}

\begin{remark}
We expect that a proof of Conjecture \ref{badset} or even the rationality statement for the exceptional values of $a$ should involve an interpolation 
of the representation theory of rational Cherednik algebras to complex rank. More precisely, recall from \cite{EGL} that if $a\ge 1$ is an integer then 
the algebra $\Lambda_{r,s,a}={\mathcal O}(X_{(a^r,1^s)}/S_{ra+s})$, carries an action of the spherical rational Cherednik algebra ${\bf e}H_{1/a}(ra+s){\bf e}$
for the symmetric group $S_{ra+s}$ with parameter $c=1/a$. More precisely, it carries a faithful action of the simple algebra 
\[
B_{r,s}(a):={\bf e}H_{1/a}(ra+s){\bf e}/I_{\rm max},
\]
where $I_{\rm max}$ is the maximal ideal in ${\bf e}H_{1/a}(ra+s){\bf e}$
(which is known to be unique). The algebra $B_{r,s}(a)$ is generated by 
$P_i={\bf e}\sum x_j^i$ and $P_i^*={\bf e}\sum D_j^i$, where $D_j$ are the Dunkl 
operators; when written in coordinates $y_k$, $P_i$ become the polynomials $P_{i,\lambda}$ 
and $P_i^*$ become the quantum integrals of the deformed Calogero--Moser system of type $A_r(s)$ with coupling constant $a$, see \cite{SV}. 

In other words, the algebra $B_{r,s}(a)$ is generated by $\Lambda_{r,s,a}$ and the algebra $\Lambda_{r,s,a}^*$ 
of quantum integrals of the deformed Calogero--Moser system (which is a commutative algebra of differential operators 
isomorphic to $\Lambda_{r,s,a}$). This definition actually makes sense for any complex $a\ne 0$, not only for positive integers, and 
gives rise to a flat family of algebras $B_{r,s}(a)$ parametrized by a complex parameter $a$. Exceptional values of $a$ are likely those for which 
this algebra ceases to be simple, and its representations degenerate (at other values of $a$, we expect that the methods 
similar to ones of \cite{EGL} should apply to establish the CM property). 

The algebras $B_{r,s}(a)$ for general $a$ are quotients of interpolations of 
\[
{\bf e}H_{1/a}(ra+s){\bf e}={\bf e}H_{\frac{r}{n-s}}(n){\bf e}
\]
to complex values of the rank $n=ra+s$ (which is a kind of toroidal Yangian, similar to the deformed double current algebras from \cite{Gu}). 
Such interpolations for the full (non-spherical) Cherednik algebras were considered in \cite{EA, E} (based on the Deligne category $\Rep(S_\nu)$). These interpolations have two parameters $c,\nu$, where $\nu$ is an interpolation of $n$, 
and the situation at hand corresponds to the hyperbola $c=\frac{r}{\nu-s}$ in the $(c,\nu)$-plane, which is a reducibility locus for 
the interpolated polynomial module, see \cite{EA}. Further degenerations occur at intersection points of this hyperbola with other 
reducibility hyperbolas, which occurs at rational values of $\nu$. This should lead to a proof of rationality of exceptional values. 
The determination of the exact set of exceptional values by this method would likely require a more detailed study of representations of the rational Cherednik algebra of 
complex rank along the lines of \cite{EA}.     

We note that an approach to deformed Calogero--Moser systems of type $A_r(s)$ essentially based on interpolation to complex rank is also proposed in \cite{SV2}. 
\end{remark} 

\section{\texorpdfstring{The case $\lambda=(a,b,c)$.}{The case lambda=(a,b,c)}}

Let us now focus on the first nontrivial case $\lambda=(a,b,c)$. 
By Theorem \ref{cher}, $X_\lambda$ is CM 
if $(a,b,c)=(a,1,1)$, $(a,a,1)$, or $(a,a,a)$ for positive integer $a$.
 
Also, we have the following result. 

\begin{proposition}\label{422} 
$X_\lambda$ is CM if $\lambda=(4,2,2)$ but is not CM if $\lambda \in \{(3,2,2), (3,3,2), (5,2,2)\}$. Thus, $X_\lambda$ is not CM for $\lambda=(3^r,2^s)$ for $r\ge 1$, $s\ge 1$, $r+s\ge 3$, and $\lambda=(5^r,2^s)$, $r\ge 1$, $s\ge 2$. 
\end{proposition} 

\begin{proof} 
The second statement follows from the first one and Proposition \ref{formneigh}. The first statement is proved by a Macaulay2 calculation 
(we do not have a theoretical proof of these facts) which is included in the file {\tt computations1.m2} and which we comment on now.

To show that $X_{(4,2,2)}$ is CM, it suffices to find a regular sequence for its coordinate ring. In the file, we produce a collection of $3$ random linear forms and test if they form a regular sequence by calculating the Hilbert series of the coordinate ring before and after quotienting by these forms. It suffices that this works for one example, and the random collection will have the desired property with high probability.

For the examples that are not CM, we calculate the ideal and observe that the numerator of the Hilbert series has negative coefficients:
\begin{align*}
h_{X_{3,2,2}}(t) &= \frac{1 + 4t + 10t^2 + 20t^3 + 35t^4 + 35t^5 + 14t^6 - 14t^7}{(1-t)^3},\\
h_{X_{3,3,2}}(t) &= \frac{1 + 5t + 15t^2 + 35t^3 + 70t^4 + 98t^5 + 70t^6 - 14t^7}{(1-t)^3},\\
h_{X_{5,2,2}}(t) &= \frac{1 + 6t + 21t^2 + 56t^3 + 126t^4 + 216t^5 - 48t^6}{(1-t)^3}.
\end{align*}

One additional comment: in the file, the calculations are set up for a finite field of size $32003$ to significantly speed up calculations. One cannot use this to deduce that the variety is not CM over a field of characteristic $0$, but the calculations can also be performed in this situation (with much more patience) and the Hilbert series stays the same.
\end{proof} 

To address the case $\lambda=(a,b,c)$, we look at $Y_\lambda^0$ 
and try to identify cases when it is not CM. 
This problem is convenient to approach by computer calculation, since 
$Y_\lambda^0$ is a surface, and $a$, $b$, and $c$ 
occur as parameters in its presentation, which can actually be taken to be nonzero complex numbers.
 
By rescaling $(a,b,c)$ as in Corollary \ref{multiple}, we may assume that $c=1$. 
Then the ring $R_{a,b}:=A_\lambda^0$ of functions on $Y_\lambda^0$ 
is generated inside $\Bbb C[x,y]$ by the homogeneous polynomials 
$$
P_i=P_{i,a|b}:=ax^i+by^i+(-ax-by)^i, \qquad i\ge 2. 
$$
(note that $P_1=0$). 
As shown in Proposition \ref{fingen1}, 
$R_{a,b}$ is a finitely generated module 
over $\Bbb C[P_2,P_3]$ outside of the lines 
\[
a+1=0,\quad b+1=0,\quad a+b=0,\quad a+b+1=0. 
\]

Let us now study the CM property of $R_{a,b}$ outside of these lines.
It is clear that in this case, $R_{a,b}$ is CM if and only if $R_{a,b}$ 
is a free $\Bbb C[P_2,P_3]$-module. 

\subsection{The special case $(a,a,1)$} 

Let us start with the case $b=a$; in this case 
$R_{a,b}=R_{a,a}=\Lambda_{2,1,a}$, generated by deformed Newton sums
$$
P_i=a(x^i+y^i)+(-a)^i(x+y)^i.
$$
Thus, we know from Theorem \ref{SVconj} that $R_{a,a}$ is CM for almost all $a$, but we 
want to prove a stronger statement: 

\begin{proposition}\label{raa} The algebra $R_{a,a}$ is CM for all $a\ne -1,-1/2$ (i.e., in all cases 
when $R_{a,a}$ is finitely generated). Moreover, for such $a$ it has the Hilbert series 
\[
\frac{1+t^4+t^5}{(1-t^2)(1-t^3)}.
\]
Therefore, the same statements hold 
for $R_{a,1}$ for $a\ne -1,-2$. Thus, $B(r,s)=\emptyset$ if $r+s=3$ (i.e., Conjecture \ref{badset} holds in this case). 
\end{proposition} 

The rest of the subsection is devoted to the proof of Proposition \ref{raa}. 

First of all, note that $R_{a,a}$ is a subring of the ring of symmetric polynomials in $x,y$, i.e., of $\Bbb C[u,v]$, 
where $u=x+y$, $v=x^2+y^2$. E.g., $P_2=av+a^2u^2$. Let us denote $P_2$ by $w$. So $v=a^{-1}w-au^2$, and $\Bbb C[u,v]=\Bbb C[u,w]$.

Now let us express $P_3$ as a linear combination of $u^3$ and $uw$.
We get 
\begin{equation}\label{p3}
P_3=\frac{3}{2}uw-\frac{1}{2}a(a+1)(2a+1)u^3. 
\end{equation}

Now we will need the following sequence of lemmas.

\begin{lemma}\label{l1} 
Every element $F \in R_{a,a}$ is a quasi-invariant, in the sense that $\frac{\partial F}{\partial x}=0$ when $x=-a(x+y)$.
\end{lemma} 

\begin{proof} By the Leibniz rule, it is enough to show that $P_i$ are quasi-invariants, 
which is easy.
\end{proof} 
 
\begin{lemma}\label{l2} Let $S$ be the algebra of all quasi-invariants.
The codimension of $S_d$ in $\Bbb C[u,w]_d$ in any positive degree $d$ is $1$. 
In other words, the Hilbert series of $S$ is 
\begin{align*}
h_S(t) &= \frac{1}{(1-t)(1-t^2)}-\frac{t}{1-t}=\frac{1+t^4+t^5}{(1-t^2)(1-t^3)}\\
&= 1+t^2+t^3+2t^4+2t^5+3t^6+3t^7+\cdots
\end{align*}
\end{lemma} 

\begin{proof} 
This follows from the fact that in every positive degree, 
quasi-invariance gives one nontrivial linear equation.
In more detail, the Hilbert series of $\Bbb C[u,v]$ is 
$\frac{1}{(1-t)(1-t^2)}$ (as we have free generators $u,v$ 
of degrees $1,2$), and to account for one linear relation
in all positive degrees, we need to subtract $\frac{t}{1-t}=t+t^2+t^3+ \cdots$.  
\end{proof} 

\begin{lemma}\label{l3} 
If for $d\ge 3$, $F\in S_d$ is divisible by $w$, then $F=wG$, where $G\in S$.
\end{lemma}

\begin{proof} 
By Lemma \ref{l2}, $\dim S_{d-2}=\dim S_d-1$, so $wS_{d-2}$ has codimension $1$ in $S_{d}$. 
But the condition of divisibility by $w$ is one nontrivial linear equation.
\end{proof} 

\begin{lemma}\label{l4} 
If $d\ge 6$ then $S_d = wS_{d-2}+P_3S_{d-3}$.
\end{lemma} 

\begin{proof} 
By Lemma \ref{l3}, $S_{d-3}$ has an element $H$ with a nonzero coefficient of $u^{d-3}$ (as not all elements are divisible by $w$).
So $P_3H\in S_d$ has a nontrivial coefficient of $u^d$, if $a$ is not $0,-1,-1/2$ (see \eqref{p3}).
The statement follows from this and Lemma \ref{l3}.
\end{proof} 

\begin{corollary} \label{genraa}
We have the following:
\begin{enumerate}[\indent \rm (i)]
\item $S$ is generated by $1,P_4,P_5$ as a module over ${\Bbb C}[P_2,P_3]$.

\item $R_{a,a}=S$. 
\end{enumerate}
\end{corollary}

\begin{proof} (i) Lemma \ref{l4} implies that $S$ is generated over $\Bbb C[P_2,P_3]$ by $S_{\le 5}$. This implies the statement, as it is easy to check that all elements 
of $S_{\le 5}$ belong to $\Bbb C[P_2,P_3](\Bbb C\cdot 1\oplus \Bbb C\cdot P_4\oplus \Bbb C\cdot P_5)$.

(ii) Since by Lemma \ref{l1}, $R_{a,a}\subset S$, this follows from (i). 
\end{proof}

Since $R_{a,a}$ clearly has rank 3 over $\Bbb C[P_2,P_3]$, 
Corollary \ref{genraa} implies that $R_{a,a}$ is freely generated over $\Bbb C[P_2,P_3]$ by 
$1,P_4,P_5$, which implies Proposition \ref{raa}. 

\subsection{\texorpdfstring{The case $a\ne b$, $a\ne 1$, $b\ne 1$}{The case a!=b, a!=1, b!=1}}

It remains to consider the case 
$a\ne b$, $a\ne 1$, $b\ne 1$. We assume that $a,b,a+b\ne 0,-1$. 

\begin{proposition}\label{bira} 
If $a\ne b$, $a\ne 1$, $b\ne 1$ then the rank of $R_{a,b}$ over $\Bbb C[P_2,P_3]$ is $6$. 
\end{proposition} 

\begin{proof} This follows from the fact that the equations $P_2(x,y)=c_2$, $P_3(x,y)=c_3$ have 6 solutions for generic 
$c_2,c_3$ (by Bezout's theorem). 
\end{proof} 

\begin{corollary} \label{Qpol}$R_{a,b}$ is CM if and only if the Hilbert polynomial $Q(t)$ of 
the finite dimensional algebra $R_{a,b}/(P_2,P_3)$ satisfies the condition $Q(1)=6$.  
\end{corollary} 

\begin{proof} $R_{a,b}$ is CM if and only if it is free over $\Bbb C[P_2,P_3]$ (of rank $6$). 
In this case, the polynomial $Q$ is the Hilbert polynomial of the generators. 
Otherwise, $Q(1)> 6$.  
\end{proof} 

Now we have the following proposition: 

\begin{proposition}\label{degrees} 
Write the Hilbert series of $R_{a,b}$ as $h_{R_{a,b}} = \frac{q(t)}{(1-t^2)(1-t^3)}$. If $a\ne b, a\ne 1, b\ne 1$ and 
$$
(a,b)\ne (3,2),(2,3),(1/2,3/2),(3/2,1/2),(1/3,2/3),(2/3,1/3),
$$ 
then   
$$
q(t)=1+t^4+t^5+t^6+t^7+t^8+Ct^9+\cdots
$$
where $C=1$ if $a\ne b+1$, $b\ne a+1$, $1\ne a+b$. 
\end{proposition} 

\begin{proof}
The proof is by a computer calculation in Macaulay2. We are trying to calculate the Hilbert series of the image of a polynomial map, which is an implicitization problem. However, that problem is too difficult, and we only need partial information. A naive approach to compute the dimension of the degree $d$ piece of $R_{a,b}$ is to compute the dimension of the linear span of the space of all monomials of (weighted) degree $d$ in the $P_i$. 

This is what is done in {\tt computations2.m2} and we first work with generic values of $a,b$ to determine the generic dimensions of these spaces. For our result, we just need $d=1,\dots,9$, in which case the dimensions are $1, 1, 1, 2, 2, 4, 4, 6, 7$ which gives 
\[
q(t) = 1+ t^4 + t^5+t^6+t^7+t^8+t^9+ \cdots
\]
for generic $a,b$.

This dimension count is actually determining the rank of a matrix: each monomial in the $P_i$ is written as a vector whose entries are the coefficients (which live in $\Bbb C[a,b]$) of its expression as a polynomial in $x,y$. To determine when this dimension can drop, we calculate the ideal of minors of each rank and decompose this ideal to get the bad conditions on $a,b$ as above. This works up to degree $8$, but degree $9$ is too hard to handle directly; instead we find a special $7 \times 7$ minor which completely factors over $\Bbb Q[a,b]$ since this gives an upper bound for the bad locus.
\end{proof} 

\begin{corollary}\label{noncm}
$R_{a,b}$ is not CM outside of the lines 
$$
a=b, a=1, b=1, a=b+1, b=a+1, 1= a+b.
$$ 
In particular, $X_{(a,b,c)}$ is not CM if $a>b>c$ and $a\ne b+c$.  
\end{corollary} 

\begin{proof} If $\{P_2,P_3\}$ does not form a regular sequence in $R_{a,b}$ then $R_{a,b}$ is not CM. If it does, then $q(t)=Q(t)$, and by Proposition \ref{degrees}, outside of the given lines we have $Q(1)\ge 7$, so the module is not free.  
\end{proof} 

\subsection{The case $a=b+1$}

Now we consider the case when $a=b+1$ (the other remaining cases, $b=a+1$ and $1=a+b$, are reduced to it by permutation of coordinates). 
Now we have the only parameter $b\ne 0$, which should not equal to $-1,-2,-1/2$.

\begin{proposition}\label{degrees1} 
The Hilbert series of $R_{3,2}$ is $\frac{q(t)}{(1-t^2)(1-t^3)}$ where 
$$
q(t)=1+t^4+t^5+t^6+t^8+t^9+t^{10} + \cdots,
$$
so $R_{3,2}$ is not CM.
\end{proposition}

This is done in \verb+computations3.m2+ by explicit calculation. 

%We do not have a bound for how many $P_i$ are needed to generate $R_{b+1,b}$ in general, but experimental calculations (included in the same file) suggest the following %statement:

Also, we have the following proposition, conjectured by the authors in the original version of this paper and proved by 
Eric Rains using MAGMA: 

\begin{proposition}\label{bplus1b}
For $b\ne 0, \pm\frac{1}{2},\pm 1,\pm 2$, $R_{b+1,b}$ is CM with  
\[
Q(t)=1+t^4+t^5+t^6+t^7+t^8. 
\]
\end{proposition}

\begin{proof}  
Let us show that $R_{b+1,b}$ is generated as a module over $\Bbb C[P_2,P_3]$ by 
\[
T_1=1,\quad T_2=P_4,\quad T_3=P_5,\quad T_4=P_6,\quad T_5=P_7,\quad T_6=P_4^2.
\]
Since this module is of generic rank $6$, this means that it is a free module, 
which implies the theorem. 

It is easy to compute (e.g., using MAGMA) 
that each of the elements $x,y,z$ is annihilated by a degree $6$ 
monic polynomial over $\Bbb C[P_2,P_3]$; 
let us denote these polynomials by $Q_x$, $Q_y$, $Q_z$. Let 
$$
Q(u):=Q_x(u)Q_y(u)Q_z(u)=\sum_{j=0}^{18} a_ju^j, 
$$
 a monic polynomial of degree $18$ (i.e., $a_{18}=1$). 
It is clear that the $P_i$ satisfy the linear recursion 
$$
\sum_{i=0}^{18}a_iP_{n+i}=0,\quad n\ge 0.
$$
So it suffices to check that $P_8,\dots, P_{17}$  belong to the $\Bbb C[P_2,P_3]$-module $M$ generated by $T_1,\dots, T_6$, 
and that this module is in fact an algebra (i.e., $T_iT_j\in M$). 
This is done using MAGMA (the computation takes less than a second).  
\end{proof} 

\begin{remark} Note that for $b=1$, $R_{b+1,b}$ is still CM, but has a different Hilbert series. 
\end{remark} 

\section{The CM property of $X_{b+c,b,c}$}

Thus, to study the CM properties of $X_{b+c,b,c}$, where $b>c\ge 1$, we need more information. 
These varieties turn out to be not CM, but the necessary information is obtained not from the ring of invariants 
(which is CM unless $b=2c$), but rather from the isotypic component of the reflection representation. 

\begin{proposition}\label{qplusrqr}
The variety $X_{b+c,b,c}$ is not CM for any positive integers $b>c$. 
\end{proposition} 

\begin{proof} Since, as we saw above, $A_{(3,2,1)}$ is not CM, 
we may assume that $b\ne 2c$. Let $\beta=b/c$. 

In this case, consider $M:=\Hom_{S_n}(\h,{\mathcal O}(X_\lambda^0))$, where 
$n=2(b+c)$ and $\h$ is the reflection representation of $S_n$. It is easy to see from standard invariant theory for $S_n$ 
that $M$ is generated over $R_{\beta+1,\beta}$ by the polynomials 
$$
T_i:=\sum_{j=1}^n x_j^iu_j, 
$$
where $\sum_{j=1}^n u_j=0$. 
(For example, this follows from the fact that the invariants of $S_n$ 
in $\Bbb C[x_1, \dots,x_n, u_1, \dots ,u_n]$ are generated by the polynomials $\sum x_j^iy_j^s$, 
a special case of the Weyl's Fundamental Theorem of Invariant Theory).  

Let $\bold x=(x_1,\dots,x_n)\in E_\lambda$. 
Let $x$ be the common value
of the first $b+c$ coordinates $x_j$, $y$ of the next $b$ coordinates, and $z$ for the last $c$ coordinates.
Let 
\[
\sum_{j=1}^{b+c} u_j=u, \quad \sum_{j=b+c+1}^{2b+c} u_j=v, \quad \sum_{j=2b+c+1}^{2(b+c)} u_j=w.
\] 
Then $u+v+w=0$, so $w=-u-v$, and
$$
T_i=x^iu+y^iv+z^iw=(x^i-z^i)u+(y^i-z^i)v,
$$
where $z=-(\beta+1)x-\beta y$.  
So we get
$$
T_i=(x^i-(-(\beta +1)x-\beta y)^i)u+(y^i-(-(\beta +1)x-\beta y)^i)v.
$$
Thus, $M$ is the module over $R_{\beta+1,\beta}$ generated by $T_i$ inside 
$R_{\beta+1,\beta}u\oplus R_{\beta+1,\beta}v$. In particular, $M$ is finitely generated. 

The rank of $M$ over $R_{\beta+1,\beta}$ is clearly $2$ (since we have two vectors $u,v$), so
rank of $M$ over $\Bbb C[P_2,P_3]$ should be $6\cdot 2=12$. 

It suffices for us to show that $M$ is not a CM module over the CM algebra $R_{\beta+1,\beta}$, i.e., that it is not free as a module over $\Bbb C[P_2,P_3]$. This will follow from the following lemma:

\begin{lemma} Set $\deg(u)=\deg(v)=0$. Writing the Hilbert series of $M$ as $h_M(t)=\frac{q(t)}{(1-t^2)(1-t^3)}$, for $a\neq 0, \pm 1, \pm 2, \pm\frac 12$, we have
$$
q(t)=t+t^2+t^3+t^4+2t^5+3t^6+3t^7+t^8+\cdots .
$$
\end{lemma} 

\begin{proof}
The argument is similar to that of Proposition~\ref{degrees}. In this case, \verb+computations4.m2+ computes $\dim M_n$, the dimension of the degree $n$ part of $M$, for $1\le n\le 8$, using the spanning set 
\[
\{P_{i_1}P_{i_2} \cdots P_{i_k}T_{i_{k+1}} \mid i_1+\cdots+i_{k+1}=n\}.
\]
We get the list $\{1, 1, 2, 3, 5, 7, 10, 11\}$ for the ranks in degrees $n=1,\dots,8$, which tells us that 
\[
q(t)=t+t^2+t^3+t^4+2t^5+3t^6+3t^7+t^8+ \cdots
\]
for generic $\beta$.

We then determine for which $\beta$ these dimensions drop, i.e., the ranks of the corresponding matrices drop. In degree up to $5$ we get that the decomposition of the ideal of minors only degenerates for $\beta=0, \pm 1$. Taking a specific minor from the matrices of degrees $6,7,8$ which splits, and using its roots as an upper bound, we see that the roots of these polynomials are contained in $\{0, \pm 1, \pm 2, \pm \frac{1}{2}, 4\}$, and a single check rules out $\beta=4$ as a degenerate value. Hence outside of these values $q(t)$ is generic.
\end{proof}

If $\{P_2,P_3\}$ is not a regular sequence, then $M$ is not CM. If it is a regular sequence, then the Hilbert polynomial $Q(t)$ of $M$ equals $q(t)$. In this case, we have just shown that $Q(1)\ge 13$, while the rank of $M$ is $12$, so $M$ cannot be CM. This proves the proposition.
\end{proof} 

\section{The non-CMness theorems} \label{sec:nonCM}

\begin{theorem}\label{noCMinv} 
Let $\lambda=(\lambda_1,\dots,\lambda_r)$ be a partition of $n$.
If $\lambda$ has at least four distinct parts, or has three distinct parts 
$a>b>c$ with $a\ne b+c$, then $X_\lambda/S_n$ and hence $X_\lambda$ is not CM.  
\end{theorem} 

\begin{proof}
The proof is by induction on $r$. For $r=3$ (base of induction), 
the statement follows from Corollary \ref{noncm}. Assume that $r\ge 4$.
If $\lambda$ has four smallest distinct parts $a>b>c>d$, 
then either $a\ne b+c$ or $a\ne b+d$, so we can remove 
either $d$ or $c$ using Proposition \ref{formneigh}, 
descending from $r$ to $r-1$ and preserving the hypothesis of the theorem. 
So it remains to consider the case when there are only three distinct parts 
$a>b>c$ and $a\ne b+c$. In this case, since $r\ge 4$, one of these parts should occur more than once, and we can 
remove it using Proposition \ref{formneigh}, descending from $r$ to $r-1$ and 
preserving the hypothesis. The theorem is proved. 
\end{proof} 

\begin{theorem}\label{noCM}
If $\lambda=(b+c,b^r,c^s)$, $b>c$, then $X_\lambda$ is not CM. 
\end{theorem} 

\begin{proof}
Using Proposition \ref{formneigh}, we can reduce to the situation when $\lambda=(b+c,b,c)$, 
in which $X_\lambda$ is not CM by Proposition \ref{qplusrqr}. 
\end{proof} 

\begin{conjecture}\label{qplusrqr1} If $\lambda=((b+c)^q,b^r,c^s)$, $b>c$, then $X_\lambda$ is not CM for $q>1$. 
Thus, $X_\lambda$ is not CM for any $\lambda$ that has at least three distinct parts.
\end{conjecture}

It suffices to prove Conjecture \ref{qplusrqr1} for the case $r=s=1$, as one can reduce to this case using Proposition \ref{formneigh}. 
However, we cannot reduce $q$ by removing a copy of $b+c$, since it equals the sum of two other parts, $b$ and $c$.  

Conjecture \ref{qplusrqr1} is supported by computational evidence for $q=2$, $r=s=1$, i.e.,  $\lambda=(b+c,b+c,b,c)$, obtained using Macaulay2
(namely, we computed that $Y_\lambda$ is not CM for $\lambda=(b+1,b+1,b,1)$ 
for random $b$). 

Let us now discuss the remaining case of two distinct parts, i.e., $\lambda=(b^r,c^s)$, $b>c$. 
Let $b/c=b'/c'$ and $\gcd(b',c')=1$. If $c'\le r$, $3\le b'\le s$, then we know from Proposition \ref{adapt3.11} 
that $X_\lambda/S_n$ and hence $X_\lambda$ is not CM. Otherwise, if $c=1$ (and $b>s$ or $b=2$), 
we know from Theorem \ref{cher} that $X_\lambda$ is CM. Finally, 
by Proposition \ref{422}, $X_{(4,2,2)}$ is CM, while $X_{(3^r,2^s)}$ is not CM 
for $r\ge 1$, $s\ge 1$, $r+s\ge 3$ and $X_{(5^r,2^s)}$ is not CM 
for $r\ge 1$, $s\ge 2$. Apart from these cases, we do not know the answer. 
In particular, the answer is not known for $\lambda=(4,2^s)$. 

The situation with $Y_\lambda$ for $\lambda=((b+1)^q,b^r,1^s)$ seems to be even more complicated. 
Computational evidence (see above) suggests a conjecture that $Y_\lambda$ is not CM if $q\ge 2$. 
However, we have the following conjecture about $q=1$: 

\begin{conjecture}\label{aplus1}
If $\lambda=(b+1,b,1^s)$ then $Y_\lambda$ is CM for generic complex $b$.
Moreover, the exceptional values of $b$ for such $\lambda$ are $b=0$ and $b=\pm p/q$, where $1\le p\le s+1$, $1\le q\le 2$ (i.e., the same as those for 
$(b^2,1^{s+1})$, see Remark \ref{exce}). 
\end{conjecture} 

Note that Conjecture \ref{aplus1} holds for $s=1$ by Proposition \ref{bplus1b}. 
It was also proved by Eric Rains for $s=2$ using MAGMA: 

\begin{proposition}\label{aplus1s2} 
Conjecture \ref{aplus1} holds for $s=2$. More specifically, 
if $b\ne 0,\pm \frac{1}{2},\pm 1,\pm 2,\pm \frac{3}{2},\pm 3$, then 
the algebra $\Bbb C[Y_{b+1,b,1,1}]$ is $CM$ with Hilbert series 
$$
\frac{1+t^5+t^6+t^7+t^8+t^9+2t^{10}+2t^{11}+t^{12}+t^{13}}{(1-t^2)(1-t^3)(1-t^4)}. 
$$
\end{proposition} 

\begin{proof} The proof is similar to the proof of Proposition \ref{bplus1b}. Let us show that $\Bbb C[Y_{b+1,b,1,1}]$ is generated as a module over $\Bbb C[P_2,P_3,P_4]$ by 
\begin{align*}
T_1=1,\quad T_2=P_5,\quad T_3=P_6,\quad T_4=P_7,\quad T_5=P_8,\quad T_6=P_9,\\
T_7=P_{10},\quad T_8=P_5^2,\quad T_9=P_{11},\quad T_{10}=P_5P_6,\quad 
T_{11}=P_5P_7,\quad T_{12}=P_5P_8.
\end{align*}
Since this module is of generic rank $12$, this means that it is a free module, 
which implies the theorem. 

It is easy to compute (e.g., using MAGMA) that each of the elements $x,y,z,w$ is annihilated by a degree $12$ 
monic polynomial over $\Bbb C[P_2,P_3,P_4]$; 
let us denote these polynomials by $Q_x$, $Q_y$, $Q_z$, $Q_w$. 
Let 
$$
Q(u):=Q_x(u)Q_y(u)Q_z(u)Q_w(u)=\sum_{j=0}^{48} a_ju^j, 
$$
 a monic polynomial of degree $48$ (i.e., $a_{48}=1$). 
It is clear that the $P_i$ satisfy the linear recursion 
$$
\sum_{i=0}^{48}a_iP_{n+i}=0,\quad n\ge 0.
$$
So it suffices to check that $P_{12},\dots,P_{47}$  belong to the $\Bbb C[P_2,P_3,P_4]$-module $M$ generated by $T_1,\dots,T_{12}$, 
and that this module is in fact an algebra (i.e., $T_iT_j\in M$). 
This is done using MAGMA (the computation takes under an hour).  
\end{proof} 

A similar conjecture applies, of course, to the case $\lambda=(b+1,b^s,1)$, as it is related to the previous case by the transformation $b\mapsto 1/b$ and rescaling. 
However, we do not know what happens for $\lambda=(b+1,b^r,1^s)$.

\end{document}